\documentclass[12pt]{article}

\usepackage{amsmath,amsthm,amsfonts,amssymb,amscd,caption,color,subcaption,cite}
\usepackage{graphicx}
\usepackage[figurename=Fig.]{caption}
\numberwithin{equation}{section}
\allowdisplaybreaks[4]

\newtheorem {Lemma}{Lemma}[section]
\newtheorem {Theorem} {Theorem}[section]

\usepackage{fullpage}

\begin{document}

\title{Spectral conditions for factor-criticality of graphs}

\author{Jin Cai\footnote{E-mail: jincai@m.scnu.edu.cn}, Bo Zhou\footnote{E-mail: zhoubo@m.scnu.edu.cn}\\
School of  Mathematical Sciences, South China Normal University,\\
Guangzhou 510631, P.R. China}

\date{}
\maketitle

\begin{abstract}
A graph $G$ is $k$-factor-critical if $G-S$ has a perfect matching for any
$k$-subset $S$ of the vertex set of $G$. In this paper, we investigate the factor-criticality of graphs with fixed minimum degree and provide sufficient conditions for such graphs
to be $k$-factor-critical in terms of spectral radius and signless Laplacian spectral radius.  
\end{abstract}

\section{Introduction}

Graphs considered in this paper are simple and undirected. Let $G$ be a graph  with vertex set $V(G)$ and edge set $E(G)$.  For any $v\in V(G)$, the neighborhood  $N_G(v)$ of $v$ is the set of vertices adjacent to $v$ in $G$, and the degree of $v$, denoted by $d_G(v)$, is the number $|N_G(v)|$.
For any $S\subset V(G)$, let $G-S$ be the subgraph of $G$ obtained by removing the vertices in $S$ (and incident edges).
For two vertex disjoint graphs $G_1$ and $G_2$,  $G_1\cup G_2$ denotes the disjoint union of $G_1$ and $G_2$, $G_1\vee G_2$ denotes the join of  $G_1$ and $G_2$, which is obtained from $G_1\cup G_2$
by adding all possible edges between any vertex of $G_1$ and any vertex of $G_2$. For positive integer $k$ and a graph $G$, $kG$ denotes the graphs consisting of $k$ vertex disjoint copies of $G$.
Denote by $K_n$ the complete graph of order $n$.

A perfect matching in a graph  is a set $M$ of
edges such that no two edges in $M$ are adjacent and every vertex is incident to an edge that belongs to $M$.  A graph $G$ is $k$-factor-critical \cite{Fa,Yu}
if $G-S$ has a perfect matching for any $k$-subset $S$
of $V(G)$. In particular, a $0$-factor-critical graph  is a graph  with a perfect matching,
a $1$-factor-critical graph is known as a factor-critical  graph, and a
$2$-factor-critical graph is known as a factor-bicritical  graph, see \cite{DR,Ga,Lo}.
By definition, if a graph of order $n$ is $k$-factor-critical, then
$k$ and $n$ have the same parity  and $0 \leq k\leq n-2$.

It has been a subject of interest to find sufficient conditions for a $k$-factor-critical graph, see, e.g. \cite{ACA,EPS}. Recently, Zhou et al. \cite{ZSZ} gave several spectral radius conditions for the existence of $(k+1)$-connected
$k$-factor-critical graphs.

For positive integers $n$ and $\delta$ with $\delta\le n-1$, let $\mathbb{G}(n,\delta)$ be the class of graphs of order $n$ with minimum degree $\delta$.  Let $G\in \mathbb{G}(n,\delta)$.
Let $u$ be a vertex of $G$ with minimum degree $\delta$. Evidently, $u$ is an isolated vertex in $G-N_G(u)$, so
$G-N_G(u)$ has no perfect matching. This shows that $G$ can not be a   $\delta$-factor-critical graph.

In this paper,  we provide sufficient conditions for graphs in $\mathbb{G}(n,\delta)$ to be
$k$-factor-critical when $k<\delta$ using spectral radius and signless Laplacian spectral radius, respectively. For a graph $G$, we denote by $\rho(G)$ and $q(G)$ the spectral radius and the signless Laplacian spectral radius of $G$, respectively. The main results are as follows.

\begin{Theorem}\label{Cai1}
Let $G\in \mathbb{G}(n,\delta)$, where $n\geq 4\delta+3$. Let $k$ be an integer in $[0,\delta)$ with $k\equiv n \pmod 2$. If $\rho(G)\geq\rho(K_\delta\vee ((\delta-k+1)K_1\cup K_{n-2\delta+k-1}))$, then $G$ is $k$-factor-critical unless $G\cong K_\delta\vee((\delta-k+1)K_1\cup K_{n-2\delta+k-1})$, where
$\rho(K_\delta\vee ((\delta-k+1)K_1\cup K_{n-2\delta+k-1}))$ is equal to the largest root of $f(x)=0$ with
\begin{align*}
f(x)& =x^3-(n+k-\delta-3)x^2-(n+\delta^2-k\delta+k-2)x\\
&\quad -2\delta^3+(n+3k-4)\delta^2+(n+3k-nk-k^2-2)\delta.
\end{align*}
\end{Theorem}

\begin{Theorem}\label{Cai2}
Let $G\in \mathbb{G}(n,\delta)$, where $n\geq 9\delta-2k+12$. Let $k$ be an integer in $[0,\delta)$ with $k\equiv n \pmod 2$. If $q(G)\geq q(K_\delta\vee((\delta-k+1)K_1\cup K_{n-2\delta+k-1}))$, then $G$ is $k$-factor-critical unless $G\cong K_\delta\vee (\delta-k+1)K_1\cup K_{n-2\delta+k-1})$, where
$q(K_\delta\vee((\delta-k+1)K_1\cup K_{n-2\delta+k-1}))$ is equal to the largest root of $g(x)=0$ with
\begin{align*}
g(x)&=x^3-(3n-\delta+2k-6)x^2\\
&\quad +(2n^2+(\delta+2k-8)n-4\delta^2+4(k-1)\delta-4k+8)x\\
&\quad -2\delta^3+(4n+4k-10)\delta^2-(2n^2+(4k-10)n+2k^2-10k+12)\delta.
\end{align*}
\end{Theorem}

\section{Preliminaries }

The adjacency matrix of a graph $G$ is defined as $\mathbf{A}(G)=(a_{uv})_{u,v\in V(G)}$, where $a_{uv}=1$ if $u$ and $v$ are adjacent, and $a_{uv}=0$ otherwise. The spectral radius of $G$, $\rho(G)$, is the largest eigenvalue of $\mathbf{A}(G)$.
Let $\mathbf{D}(G)$ be the degree diagonal matrix of $G$. The matrix $\mathbf{Q}(G)=\mathbf{D}(G)+\mathbf{A}(G)$ is the signless Laplacian matrix of $G$. The
signless Laplacian spectral radius of $G$, $q(G)$,
is the largest eigenvalue of $\mathbf{Q}(G)$.

We need some lemmas in the proofs.

By  Perron-Frobenius theorem, we have the following lemma.

\begin{Lemma} \label{inequit} 
Let $G$ be a graph and $u$ and $v$ two distinct vertices that are not adjacent in $G$. If $G+uv$ is connected, then $\rho(G+uv)>\rho(G)$ and $q(G+uv)>q(G)$.
\end{Lemma}

Let $V(G)=V_1\cup \dots\cup V_s$ be a partition of $V(G)$. For $1\le i<j\le s$, set $\mathbf{B}_{ij}$ denotes  the submatrix of $\mathbf{H}(G)\in \{\mathbf{A}(G), \mathbf{Q}(G)\}$ with rows corresponding to vertices in $V_i$ and columns corresponding to vertices in $V_j$. The quotient matrix of $\mathbf{H}(G)$ with respect to the partition $V_1\cup \dots \cup V_s$ is the matrix $\mathbf{B}=(b_{ij})$, where $b_{ij}$ equals to the average row sums of $\mathbf{B}_{ij}$. If $\mathbf{B}_{ij}$ has constant row sum, then we say $B$ is an equitable quotient matrix (with respect to the above partition of $V(G)$).
The following lemma is an immediate consequence of \cite[Lemma 2.3.1]{BH}

Denote by $\lambda(\mathbf{M})$ the spectral radius of a square nonnegative matrix $\mathbf{M}$, which is an eigenvalue of $\mathbf{M}$.

\begin{Lemma} \label{equit} \cite{BH}
If $\mathbf{B}$ is an equitable quotient matrix of $\mathbf{H}(G)\in \{\mathbf{A}(G), \mathbf{Q}(G)\}$, then  the eigenvalues of $\mathbf{B}$ are also eigenvalues of $\mathbf{H}(G)$, and  $\lambda(\mathbf{H}(G))$ is equal to the largest eigenvalue of $\mathbf{B}$.
\end{Lemma}

Denote by  $o(G)$  the number of odd components of a graph $G$.

\begin{Lemma} \label{fc} \cite{Fa}
A graph $G$ is $k$-factor-critical if and only if $o(G-S)\leq |S|-k$ for every $S\subset V(G)$ with $|S|\geq k$.
\end{Lemma}


Given a connected graph $G$ and  $\alpha=0,1$, we denote by $\lambda_\alpha(G)$ the spectral radius of the matrix $\alpha\mathbf{D}(G)+\mathbf{A}(G)$. By Perron-Frobenius theorem, there is a unique unit positive eigenvector corresponding to $\lambda_\alpha(G)$, which is called the Perron vector.

For the graph $G=K_s\vee (K_{n_1}\cup \dots \cup K_{n_t})$ with $s\ge 1$, the out copy of $G$ denotes $K_s$ and the $i$-th graph in
$K_{n_1}\cup \dots \cup K_{n_t}$ is called the $i$-th inner copy of $G$, where $i=1,\dots, t$.
If $X$ is the Perron vector of $\alpha \mathbf{D}(G)+\mathbf{A}(G)$ with $\alpha=\{0,1\}$, then
entries of $X$  at any two vertices in the out copy or the same inner copy are equal by symmetry.

\begin{Lemma} \label{SP}
For positive integers $s, t, n_1, \dots n_t$ with
$t\ge 2$ and $n_1\le \dots \le n_t$, let $G=K_s\vee (K_{n_1}\cup \dots \cup K_{n_t})$ and $X$ be the Perron vector of $\alpha \mathbf{D}(G)+\mathbf{A}(G)$ with $\alpha=\{0,1\}$ , where for $i=1,\dots, t$, $x_i$ is the entry of $X$ at any vertex of the $i$-th inner copy of $G$. Then $x_i\le x_{i+1}$ for $i=1,\dots, t-1$.
\end{Lemma}

\begin{proof}
Denote by $x_0$ the entry of $X$ at any vertex of the out copy of $G$. Then
\[
(\lambda_\alpha(G)-\alpha(n_i-1+s)+1-n_i)x_i=sx_0=(\lambda_{\alpha}(G)-\alpha(n_{i+1}-1+s)+1-n_{i+1})x_{i+1}
\]
for $i=1,\dots, n_t-1$.
As $n_i\le n_{i+1}$ and $\lambda_{\alpha}(G)> \lambda_{\alpha}(K_{n_{i+1}})=(\alpha+1)(n_{i+1}-1)$, one gets
$x_i\le x_{i+1}$.
\end{proof}

The following is known in \cite{WXH} when $\alpha=0$ and in \cite{HZ} when $\alpha=1$, see also \cite{NP}.

\begin{Lemma} \label{GE} Let $G$ be a connected graph and $u$ and $v$ be two  vertices of $G$.
Let $X$ be the Perron vector of $\alpha \mathbf{D}(G)+\mathbf{A}(G)$ with $x_u\ge x_v$, where $\alpha=\{0,1\}$. Suppose that $N_G(v)\setminus (N_G(u)\cup \{u\})\ne \emptyset$. Then for any nonempty $N\subseteq N_G(v)\setminus (N_G(u)\cup \{u\})$,
\[
\lambda_\alpha(G-\{vw: w\in N\}+\{uw: w\in N\})>\lambda_{\alpha}(G).
\]
\end{Lemma}

\section{Results}

For integers $\delta, k, n$ with $\delta>k$ and  $n> 2\delta-k+1$,
$H(n,\delta,k)=K_\delta\vee ((\delta-k+1)K_1\cup K_{n-2\delta+k-1})$.

\begin{Lemma} \label{h2} For positive integers $\delta, s,k, n$ with $n\geq s+(s-k+2)(\delta-s+1)$ and $k\le s<\delta$,
let $H'_{s}=K_s\vee ((s-k+1)K_{\delta-s+1}\cup K_{n-s-(s-k+1)(\delta-s+1)})$. Then
\[
\rho(H'_{s})<\rho(H(n,\delta,k)) \mbox{ and } q(H'_{s})<q(H(n,\delta,k)).
\]
\end{Lemma}

\begin{proof}
Partition $V(H'_{s})$ into $S\cup V_1\cup V_2$, where $V_1=V((s-k+1)K_{\delta-s+1})$ and $V_2=V(K_{n-s-(s-k+1)(\delta-s+1)})$.
Assume that
\[
S=\{u_1,u_2,\dots,u_s\}, V_1=\{v_1,v_2,\dots,v_{(s-k+1)(\delta-s+1)}\},
\]
\[
V_2=\{w_1,w_2,\dots,w_{n-s-(s-k+1)(\delta-s+1)}\}.
\]
Let $X$ ($Y$, respectively) be the Perron vector of $\mathbf{A}(H'_s)$ ($\mathbf{Q}(H'_s)$). By symmetry, $X$ ($Y$, respectively) takes the same value, say $x_1, x_2$ and $x_3$ ($y_1, y_2$ and $y_3$, respectively) at vertices of $S, V_1$ and $V_2$, respectively.
Evidently, $n-s-(s-k+1)(\delta-s+1)\ge \delta-s+1$.
By Lemma~\ref{SP}, we have $x_2\leq x_3$ ($y_2\le y_3$, respectively).

Let  $H'\cong H'_s+E_1-E_2$, where
\[
E_1=\{w_iv_j:1\leq i\leq\delta-s, 1\leq j\leq(s-k+1)(\delta-s+1)\},
\]
\[
E_2=\{v_iv_j:1\leq i<j\leq\delta-s+1\}.
\]
Then  $H'\cong K_{\delta}\vee((\delta-s+1)K_1\cup(s-k)K_{\delta-s+1}\cup K_{n-2\delta+s-(s-k)(\delta-s+1)-1})$.
Thus
\begin{align*}
&\quad X^T(\mathbf{A}(H')-\mathbf{A}(H'_s))X\\
&= 2\sum_{uv\in E_1}{x_ux_v}-2\sum_{uv\in E_2}{x_ux_v}\\
&=2 \sum_{i=1}^{\delta-s} \sum_{j=1}^{(s-k+1)(\delta-s+1)}{x_{w_i}x_{v_j}}-2\sum_{i=1}^{\delta-s} \sum_{j=i+1}^{\delta-s+1}{x_{v_i}x_{v_j}}\\
&=2(\delta-s)(s-k+1)(\delta-s+1)x_2x_3-(\delta-s+1)(\delta-s)x^2_2\\
&=(\delta-s+1)(\delta-s)x_2(2(s-k+1)x_3-x_2)\\
&>0.
\end{align*}
Now Rayleigh's principle implies that
$\rho(H'_s)=X^T\mathbf{A}(H'_s)X<X^T\mathbf{A}(H')X\le \rho(H')$.
By Lemmas  \ref{SP} and \ref{GE} (whether $n-2\delta+s-(s-k)(\delta-s+1)-1\ge \delta-s+1$ or not), we have $\rho(H')\leq\rho(H(n,\delta,k))$. It thus follows that $\rho(H'_{s})<\rho(H(n,\delta,k))$.

Similarly, one also has
\begin{align*}
&\quad Y^T(\mathbf{Q}(H')-\mathbf{Q}(H'_{s}))Y\\
&= \sum_{uv\in E_1}{(y_u+y_v)^2}-\sum_{uv\in E_2}{(y_u+y_v)^2}\\
&= \sum_{i=1}^{\delta-s} \sum_{j=1}^{(s-k+1)(\delta-s+1)}{(y_{w_i}+y_{v_j})^2}-\sum_{i=1}^{\delta-s} \sum_{j=i+1}^{\delta-s+1}{(y_{v_i}+y_{v_j})^2}\\
&=(\delta-s)(s-k+1)(\delta-s+1)(y_2+y_3)^2-(\delta-s+1)(\delta-s)(2y_2)^2\\
&>0,
\end{align*}
and by Rayleigh's principle,
$q(H'_s)=Y^T\mathbf{Q}(H'_s)Y<Y^T\mathbf{Q}(H')Y\le q(H')$.
By Lemmas \ref{SP} and \ref{GE}, we have $q(H')\leq q(H(n,\delta,k))$. Thus  $q(H'_{s})<q(H(n,\delta,k))$.
\end{proof}

\begin{Lemma} \label{JISUAN} For integers $\delta, k, n$ with $\delta>k$ and  $n> 2\delta-k+1$,
$\rho(H(n,\delta,k))$ is equal to the largest root of $f(x)=0$, where
\begin{align*}
f(x)& =x^3-(n+k-\delta-3)x^2-(n+\delta^2-k\delta+k-2)x\\
&\quad -2\delta^3+(n+3k-4)\delta^2+(n+3k-nk-k^2-2)\delta.
\end{align*}

$q(H(n,\delta,k))$ is equal to the largest root of $g(x)=0$, where
\begin{align*}
g(x)&=x^3-(3n-\delta+2k-6)x^2\\
&\quad +(2n^2+(\delta+2k-8)n-4\delta^2+4(k-1)\delta-4k+8)x\\
&\quad -2\delta^3+(4n+4k-10)\delta^2-(2n^2+(4k-10)n+2k^2-10k+12)\delta.
\end{align*}
\end{Lemma}

\begin{proof} By Lemma \ref{equit},
$\rho(H(n,\delta,k))$ is equal to the largest eigenvalue of
\[
\begin{pmatrix}
\delta-1 & \delta-k+1 & n-2\delta+k-1\\
\delta & 0  & 0\\
\delta & 0 & n-2\delta+k-2
\end{pmatrix}
\]
and $q(H(n,\delta,k))$  is equal to the largest eigenvalue of
\[
\begin{pmatrix}
n+\delta-2 & \delta-k+1 & n-2\delta+k-1\\
\delta & \delta  & 0\\
\delta & 0 & 2n-3\delta+2k-4
\end{pmatrix}.
\]
Now the result follows.
\end{proof}

\subsection{Spectral radius}

\begin{Lemma} \label{h1} For positive integers $\delta, s,k, n$ with $n\geq4\delta+3$ and $\max\{\delta+1, k\}\le s\leq\frac{n+k-2}{2}$,
\[
\rho(K_s\vee ((s-k+1)K_1\cup K_{n-2s+k-1}))<\rho(H(n,\delta,k)).
\]
\end{Lemma}

\begin{proof} Let $H_s=K_s\vee ((s-k+1)K_1\cup K_{n-2s+k-1})$.
Partition $V(H_{s})$ into $S\cup V_1\cup V_2$, where $V_1=V((s-k+1)K_1)$ and $V_2=V(K_{n-2s+k-1})$.
It is easy to see that this partition is an equitable partition.  The corresponding quotient matrix is
 \[
\mathbf{B}=
 \begin{pmatrix}
     s-1 &  s-k+1 & n-2s+k-1 \\
     s & 0& 0\\
     s & 0& n-2s+k-2
 \end{pmatrix}.
 \]
By a simple calculation, the characteristic polynomial of $\mathbf{B}$ is
\begin{align*}
f_s(x)=&x^3-(n+k-s-3)x^2-(n+s^2-ks+k-2)x-2s^3\\
&+(n+3k-4)s^2+(n+3k-nk-k^2-2)s.
\end{align*}
By Lemma \ref{equit},  $\rho(H_{s})$ is equal to the largest root of the equation $f_s(x)=0$.
Similarly,
$\rho(H(n,\delta,k))$ is equal to the largest root of $f_\delta(x)=0$. It is easy to see that
\begin{align*}
\frac{f_s(x)-f_\delta(x)}{s-\delta}=h(x)&:=x^2-(s+\delta-k)x-2(\delta^2+\delta s+s^2)\\
&\quad  +(n+3k-4)(s+\delta)+n-k^2-nk+3k-2.
\end{align*}
The symmetry axis of $h(x)$ is $x=x_0:=\frac{s+\delta-k}{2}<s$.
Since $n\geq2s-k+2$, we have $s\le n-s+k-2<n-\delta+k-2$. Thus $x_0<n-\delta+k-2$.
It then follows  that  $h(x)$ is strictly increasing for $x\in[n-\delta+k-2,+\infty)$.
Thus
\begin{align*}
h(x) & \geq h(n-\delta+k-2)\\
&=-2s^2+(2k-\delta-2)s+n^2+(2k-3)n-2\delta^2+(k-2)\delta+2k^2-3k+2.
\end{align*}
As $\delta>k$, we have $\frac{2k-\delta-2}{4}<\delta+1$. Note also that
$\delta+1\leq s\leq\frac{n+k-2}{2}$ and $n\geq 4\delta+3$. Thus
\begin{align*}
h(x)
&\geq -2\left(\frac{n+k-2}{2}\right)^2+(2k-\delta-2)\frac{n+k-2}{2}+n^2+(2k-3)n\\
&\quad -2\delta^2+(k-2)\delta+2k^2-3k+2\\
&=\frac{1}{2}n^2+\left(2k-\frac{1}{2}\delta-2\right)n-2\delta^2+\left(\frac{1}{2}k-1\right)\delta+\frac{5}{2}k^2-4k+2\\
&\geq\frac{1}{2}(4\delta+3)^2+\left(2k-\frac{1}{2}\delta-2\right)(4\delta+3)-2\delta^2+\left(\frac{1}{2}k-1\right)\delta+\frac{5}{2}k^2-4k+2\\
& =4\delta^2+\frac{17}{2}\delta k+\frac{3}{2}\delta+\frac{5}{2}k^2+2k+\frac{1}{2}\\
& >0.
\end{align*}
Thus $f_s(x)>f_\delta(x)$ for $x\in[n-\delta+k-2,+\infty)$.  By Lemma \ref{inequit}, we have
$\rho(H(n,\delta,k))>\rho(K_{n-\delta+k-1})=n-\delta+k-2$,
so
$f_s(\rho(H(n,\delta,k)))>f_\delta(\rho(H(n,\delta,k)))$,
implying that $\rho(H_s)<\rho(H(n,\delta,k))$.
\end{proof}

Now we are ready to prove Theorem  \ref{Cai1}.

\begin{proof}[Proof of Theorem  \ref{Cai1}]
Suppose contradiction that $G$ is not $k$-factor-critical. By Lemma \ref{fc}, there exists a vertex subset $S\subset V(G)$ such that $o(G-S)> |S|-k$, where $|S|\geq k$. Let $|S|=s$. Since $k\equiv n\equiv s+o(G-S) \pmod 2$, one gets $o(G-S)\equiv k-s \pmod 2$. Thus  $o(G-S)\geq s-k+2$. Since $n\geq s+o(G-S)$, we have $n\geq2s-k+2$, i.e., $s\le \frac{1}{2}(n+k-2)$.

Let $t=s-k+2$.
Let $n_1\le \dots \le n_{t-1}$ be the orders
the odd components of $G-S$ with the first $t-1$ smallest orders, and let $n_t=n-s-n_1-\dots-n_{t-1}$.
Then $G$ is a spanning subgraph of the graph $G':=K_s\vee (K_{n_1}\cup \dots \cup K_{n_t})$.
By Lemma \ref{inequit},
\[
\rho(G)\le  \rho(G')
\]
with equality if and only if $G\cong G'$.
Let
$H_s=K_s\vee((s-k+1)K_1\cup K_{n-2s+k-1})$ if $s\ge 1$.

\noindent
{\bf Case 1.} $s<\delta$.

If $s=0$, then $k=0$,  $t=2$, $n_1+n_2=n$ and $n_2\ge n_1\ge \delta+1$, so $\rho(G)\le \rho(G')=\rho(K_{n_2})=n_2-1=n-n_1-1\le n-\delta-2=\rho(K_{n-\delta-1})<\rho(H(n,\delta,k))$ by Lemma \ref{inequit}, which is a contradiction.

Suppose that $s\ge 1$. Then $n_1\ge \delta+1-s$ and
$n-s-(s-k+1)(\delta-s+1)\ge n_t\ge \delta-s+1$.
\[
n\geq \delta+(s-k+2)(\delta-s+1)
\]

By Lemmas  \ref{SP} and \ref{GE}, we have
\[
\rho(G')\le \rho(K_s\vee((s-k+1)K_{\delta+1-s}\cup K_{n-s-(s-k+1)(\delta+1-s)}))
\]
with equality if and only if $G'\cong K_s\vee((s-k+1)K_{\delta+1-s}\cup K_{n-s-(s-k+1)(\delta+1-s)})$.
Thus
\[
\rho(G)\le \rho(K_s\vee((s-k+1)K_{\delta+1-s}\cup K_{n-s-(s-k+1)(\delta+1-s)}))
\]
with equality if and only if $G\cong K_s\vee((s-k+1)K_{\delta+1-s}\cup K_{n-s-(s-k+1)(\delta+1-s)})$.
By Lemma \ref{h2},
\[
\rho(K_s\vee((s-k+1)K_{\delta+1-s}\cup K_{n-s-(s-k+1)(\delta+1-s)}))<\rho(H(n,\delta,k)),
\]
which is a contradiction.

\noindent
{\bf Case 2.} $s\ge \delta$.

By Lemmas \ref{SP} and \ref{GE}, we have
$\rho(G')\le \rho(H_s)$
with equality if and only if $G'\cong H_s$. Thus
$\rho(G)\le \rho(H_s)$
with equality if and only if $G\cong H_s$.
By the assumption, $\rho(G)\ge \rho(H(n,\delta,k))$. Thus
\[
\rho(H(n,\delta,k))\le \rho(G) \le \rho(H_s).
\]
If $s\ge \delta +1$, then we have by Lemma \ref{h1} that
$\rho(H(n,\delta,k))\le \rho(H_s)<\rho(H(n,\delta,k))$, which is a contradiction.
Thus  $s=\delta$ and $\rho(G)=\rho(H(n,\delta,k))$, so $G\cong H(n,\delta,k)$, which is also a contradiction.

The result follows by Lemma \ref{JISUAN}.
\end{proof}

If $k\ge 1$, then $\delta\ge 2$ in Theorem \ref{Cai1}, and it is easy to check that the condition on $n$
may be weakened to $n\geq 4\delta+1$.

\subsection{Signless Laplacian spectral radius}

%
%

\begin{Lemma} \label{h3} For positive integers $\delta, s,k, n$ with $n\geq 9\delta-2k+12$, $\delta+1\leq s\leq\frac{n+k-2}{2}$,
\[
q(K_s\vee ((s-k+1)K_1\cup K_{n-2s+k-1}))<q(H(n,\delta,k)).
\]
\end{Lemma}

\begin{proof}  Let $H_s=K_s\vee ((s-k+1)K_1\cup K_{n-2s+k-1})$.
Partition $V(H_{s})$ into $S\cup V_1\cup V_2$, where $V_1=V((s-k+1)K_1)$ and $V_2=V(K_{n-2s+k-1})$.
It is easy to see that this partition is an equitable partition.  The corresponding quotient matrix is
 \[
\mathbf{B}=
 \begin{pmatrix}
     n+s-2 & s-k+1 & n-2s+k-1 \\
     s & s & 0 \\
     s & 0 & 2n-3s+2k-4
 \end{pmatrix}.
 \]
By a simple calculation, the characteristic polynomial of $\mathbf{B}$ is
\begin{align*}
f_s(x)&=x^3-(3n+2k-s-6)x^2\\
&\quad +(2n^2+sn+2kn-8n-4s^2+4(k-1)s-4k+8)x\\
&\quad -2s^3+(4n+4k-10)s^2\\
&\quad -(2n^2+4kn-10n+2k^2-10k+12)s.
\end{align*}
By Lemma \ref{equit},  $q(H_{s})$ is equal to the largest root of the equation $f_s(x)=0$.
Similarly,
$q(H(n,\delta,k))$ is equal to the largest root of $f_\delta(x)=0$. Then
\begin{align*}
\frac{f_s(x)-f_\delta(x)}{s-\delta}=h(x)&:=x^2+(n-4s-4\delta+4k-4)x\\
&\quad -2(\delta^2+\delta s+s^2)+(4n+4k-10)(s+\delta)\\
&\quad -2n^2-4kn+10n-2k^2+10k-12.
\end{align*}
The symmetry axis of $h(x)$ is $x=x_0:=-\frac{1}{2}n+2s+2\delta-2k+2<-\frac{1}{2}n+4s-2k+2$.
Since $n\geq2s-k+2$, we have $2s\le n+k-2<n+\delta-2$. Thus
\begin{align*}
x_0 & <\frac{3}{2}n+2\delta-2k-2\\
    & \le 2n-\frac{1}{2}(9\delta-2k+12)+2\delta-2k-2\\
    & <2n-2\delta+2k-4.
\end{align*}
It then follows  that  $h(x)$ is strictly increasing for $x\in[2n-2\delta+2k-4,+\infty)$.
As $\delta+1\leq s\leq\frac{n+k-2}{2}$ and $n\geq 9\delta-2k+12$, we have
\begin{align*}
h(x) & \geq h(2n-2\delta+2k-4)\\
&=-2s^2-(4n-6\delta+4k-6)s+4n^2-(14\delta-14k+18)n+10\delta^2\\
&\quad -(20k-30)\delta+10k^2-30k+20\\
&\geq-2\left(\frac{n+k-2}{2}\right)^2-(4n-6\delta+4k-6)\frac{n+k-2}{2}\\
&\quad +4n^2-(14\delta-14k+18)n\\
&\quad +10\delta^2-(20k-30)\delta+10k^2-30k+20\\
&=\frac{3}{2}n^2+(12k-14\delta-9)n+\frac{25}{2}\delta^2-(22k-29)\delta+10k^2-26k+12\\
&\geq\frac{3}{2}(9\delta-2k+12)^2+(12k-14\delta-9)(9\delta-2k+12)+\frac{25}{2}\delta^2-(22k-29)\delta\\
&\quad +10k^2-26k+12\\
& =8\delta^2+60\delta k+104\delta-8k^2+64k+120\\
& >0.
\end{align*}
Thus $f_s(x)>f_\delta(x)$ for $x\in[2n-2\delta+2k-4,+\infty)$.  By Lemma \ref{inequit}, we have
$q(H(n,\delta,k))>q(K_{n-\delta+k-1})=2n-2\delta+2k-4$,
so
$f_s(q(H(n,\delta,k)))>f_\delta(q(H(n,\delta,k)))$,
implying that $q(H_s)<q(H(n,\delta,k))$.
\end{proof}

Now, we prove Theorem  \ref{Cai2}.

\begin{proof}[Proof of Theorem  \ref{Cai2}]
Suppose contradiction that $G$ is not $k$-factor-critical. By Lemma \ref{fc}, there exists a vertex subset $S\subset V(G)$ such that $o(G-S)> |S|-k$, where $|S|\geq k$. Let $|S|=s$. Since $k\equiv n\equiv s+o(G-S) \pmod 2$, one gets $o(G-S)\equiv k-s \pmod 2$. Thus  $o(G-S)\geq s-k+2$. Since $n\geq s+o(G-S)$, we have $n\geq2s-k+2$.

Let $t=s-k+2$.
Let $n_1\le \dots \le n_{t-1}$ be the orders
the odd components of $G-S$ with the first $t-1$ smallest orders, and let $n_t=n-s-n_1-\dots-n_{t-1}$.
Then $G$ is a spanning subgraph of the graph $G':=K_s\vee (K_{n_1}\cup \dots \cup K_{n_t})$.
By Lemma \ref{inequit},
$q(G)\le  q(G')$
with equality if and only if $G\cong G'$. Let
$H_s=K_s\vee((s-k+1)K_1\cup K_{n-2s+k-1})$.

Suppose first that $s\leq \delta-1$.
If $s=0$, then $k=0$ and $t=2$, so we have $q(G)\le q(G')\le q(K_{n-\delta-1})<q(H(n,\delta,k))$ by Lemma \ref{inequit}, a contradiction. Suppose that $s\ge 1$.
Then $n_1-1+s\ge \delta$, i.e., $n_1\ge \delta+1-s$.
By Lemmas  \ref{SP} and \ref{GE}, we have
$q(G)\le q(G')\le q(K_s\vee((s-k+1)K_{\delta+1-s}\cup K_{n-s-(s-k+1)(\delta+1-s)}))$
with equality if and only if $G\cong K_s\vee((s-k+1)K_{\delta+1-s}\cup K_{n-s-(s-k+1)(\delta+1-s)})$.
By Lemma \ref{h2},
\[
q(K_s\vee((s-k+1)K_{\delta+1-s}\cup K_{n-s-(s-k+1)(\delta+1-s)}))<q(H(n,\delta,k)),
\]
So $q(G)<q(H(n,\delta,k))$,
which is a contradiction.

Suppose next that $s\ge \delta$. By Lemmas \ref{SP} and \ref{GE}, we have
$q(G)\le q(G')\le q(H_s)$ with equality if and only if $G\cong H_s$.
By the assumption, $q(H(n,\delta,k))\le q(G)$. By Lemma \ref{h3}, $s=\delta$, so $G=H(n,\delta,k)$, which is a contradiction.

The result follows by Lemma \ref{JISUAN}.
\end{proof}

If $k\ge 1$, then $\delta\ge 2$ in Theorem \ref{Cai2}, and it is easy to check that the condition on $n$
may be weakened to $n\geq 9\delta-2k$.

%
%
%
%
%
%

\end{document}